\documentclass[a4paper,twoside]{article}
\usepackage{amsfonts}
\usepackage[all]{xy}
\usepackage[dvips]{graphics}
\usepackage{amssymb}
\usepackage{mathrsfs}
\usepackage[reqno]{amsmath}
\usepackage{amsthm}
\usepackage{makeidx}
\usepackage{color}
\usepackage{graphicx}
\usepackage{subfigure}
\usepackage{multicol}
\usepackage{fancyhdr}
\usepackage[colorlinks,linkcolor=blue,anchorcolor=blue,citecolor=blue]{hyperref}
\usepackage[utf8]{inputenc}

 \newtheorem{theorem}{\sc\bf Theorem}[section]
 
 \newtheorem{corollary}[theorem]{\sc\bf Corollary}
 
 \newtheorem{proposition}[theorem]{\sc\bf Proposition}
 \newtheorem{definition}[theorem]{\sc\bf Definition}
 \newtheorem{remark}[theorem]{\sc\bf Remark}
 \newtheorem{question}[theorem]{\sc\bf Question}

  \numberwithin{equation}{section}

\makeatletter
\def\@cite#1#2{#1\if@tempswa , #2\fi}
\makeatother

\title{{\bf On the spectral identities and fundamental properties of one-sided Drazin inverses in Banach algebras}\thanks{This work has been supported by the Natural Science Foundation of Fujian Province, China (Grant No. 2022J01104).}}
\author {Kai \textsc{Yan}~\thanks{E-mail address: yklolxj@163.com}
\\ \small (School of Mathematics and Statistics, Fuzhou University,
Fuzhou
350108,\\ \small P.R. China)\\
}

\begin{document}

\date{}
\maketitle

\large

\linespread{1.06}

\begin{quote}
{\bf Abstract.} We establish several fundamental properties of one-sided (generalized) Drazin inverses in Banach algebras, including intertwining properties and reverse order laws. In particular, we introduce the concepts of one-sided strongly $\pi$-regularity, which is shown to be equivalent to one-sided Drazin invertibility. By utilizing the Jacobson\textquoteright s lemma for one-sided regularity, we prove the Jacobson\textquoteright s lemma for one-sided (generalized) Drazin invertibility. These results allow us to derive the spectral identities for one-sided (generalized) Drazin invertible spectra in Banach algebras, as well as the spectral identities for Fredholm type operators acting on Banach spaces.
\end{quote}

\section{Introduction and preliminary}

The concepts of Drazin inverse and generalized Drazin inverse were originally introduced by Drazin in 1958 and Koliha in 1996, respectively. Throughout this paper, $\mathcal{S}$, $\mathcal{R}$ and $\mathcal{A}$ denote a semigroup, a ring and a Banach algebra, respectively. An element $a\in \mathcal{S}$ is said to be \emph{Drazin invertible} if there exist an element $x\in \mathcal{S}$ and an integer $j$ such that
$$ax=xa,~~x^2a=x,~~a^{j+1}x=a^j.$$
The minimal non-negative integer $j$ is referred to as the Drazin index of $a$. The Drazin inverse has found applications in the theory of finite Markov chains, linear systems of differential equations and many other fields of mathematics, see [\cite{Campbell}]. By extending nilpotent to quasi-nilpotent, Koliha [\cite{Koliha1}] introduced the generalized Drazin inverse in Banach algebras. An element $a\in \mathcal{A}$ is said to be \emph{generalized Drazin invertible} if there exists an element $x\in \mathcal{A}$ such that
$$ax=xa,~~x^2a=x,~~a-axa\in \mathcal{A}^{qnil}.$$
where $\mathcal{A}^{qnil}$ denotes all the quasi-nilpotent elements in $\mathcal{A}$. In this paper, $\sigma(a)$ and $r(a)$ denote the spectrum and the spectral radius of $a\in \mathcal{A}$, respectively. The set of all quasi-nilpotent elements in Banach algebra $\mathcal{A}$ is defined by $\mathcal{A}^{qnil}=\{a\in\mathcal{A}: r(a)=0\}.$
 Using the classical Gelfand\textquoteright s representation theorem, we have $r(ab)\leq r(a)r(b)$ when $a, b$ are commuting elements in $\mathcal{A}$.

An element $a\in \mathcal{R}$ is called \textit{strongly $\pi$-regular} if there exist elements $x,y\in \mathcal{R}$ and integers $p, q$ such that
$$xa^{p+1}=a^p \makebox{~and~} a^{q+1}y=a^q,$$
see [\cite{Azu}]. This definition can be restated as follows: $a^n\in a^{n+1}\mathcal{R}\cap\mathcal{R}a^{n+1}$ for some integer $n\geq 1$. In [\cite{Azu}], Azumaya proved that $a\in \mathcal{R}$ is strongly $\pi$-regular if and only if there exists an element $b\in \mathcal{R}$ such that $ab=ba$ and $a^n=a^{n+1}b$ for some integer $n\geq 1$. Drazin [\cite{Drazin}] proved the following result (Theorem \ref{6}), which Lam and Nielsen referred to as the \textquotedblleft Azumaya realization\textquotedblright ~in [\cite{Lam}]. The \textquotedblleft Azumaya realization\textquotedblright~ plays a significant role in studying Jacobson\textquoteright s lemma for Drazin inverse. One can check [\cite{Lam}, \cite{Yan2}] for more details.

\begin{theorem} \label{6}\rm{([\cite{Drazin}])} Let $\mathcal{R}$ be a ring and $a\in \mathcal{R}$. Then $a$ is strongly $\pi$-regular if and only if $a$ is Drazin invertible.
\end{theorem}

Recently, several researchers consider the one-sided analogous of Drazin inverse in Banach algebras and rings, see [\cite{Berkani}, \cite{Ren}, \cite{Yan1}]. In [\cite{Berkani}], Berkani defined the one-side Drazin invertibility by utilizing the concept of p-invertibility, while Ren and Jiang [\cite{Ren}] defined the one-side Drazin invertibility in terms of annihilator. These two kinds of definitions are equivalent. In this paper, we primarily employ the notions of one-side Drazin invertibility introduced in [\cite{Yan1}], as the concepts in [\cite{Yan1}] impose fewer restrictions and are therefore easier to handle.

\begin{definition} \rm{([\cite{Yan1}])} \emph{An element $a\in \mathcal{A}$ is called left Drazin invertible with index $j$, if there exists an element $x\in \mathcal{A}$ such that
$$axa=xa^2, ~~x^2a=x,~~xa^{j+1}=a^j.$$
Such $x$ is called the left Drazin inverse of $a$. The minimal non-negative integer $j$ is referred to the left Drazin index of $a$. Dually, an element $a\in \mathcal{A}$ is called right Drazin invertible with index $j$, if there exists an element $y\in \mathcal{A}$ such that
$$aya=a^2y, ~~ay^2=y, ~~a^{j+1}y=a^j.$$
Such $y$ is called the right Drazin inverse of $a$. The minimal non-negative integer $j$ is referred to the right Drazin index of $a$.
}
\end{definition}

On taking left (resp. right) Drazin index $j=1$, we can give the definitions of \emph{left~(resp.~right)~group~invertibility} in Banach algebras. Next, we introduce the notions of one-sided generalized Drazin invertible in Banach algebras.

\begin{definition} \rm{([\cite{Yan1}])} \emph{An element $a\in \mathcal{A}$ is called left generalized Drazin invertible, if there exists an element $x\in \mathcal{A}$ such that
$$axa=xa^2, ~~x^2a=x, ~~axa-a\in \mathcal{A}^{qnil}.$$
Such $x$ is called the left generalized Drazin inverse of $a$. Dually, an element $a\in \mathcal{A}$ is called right generalized Drazin invertible, if there exists an element $y\in \mathcal{A}$ such that
$$aya=a^2y, ~~ay^2=y, ~~aya-a\in \mathcal{A}^{qnil}.$$
Such $y$ is called the right generalized Drazin inverse of $a$.
}
\end{definition}

In [\cite{Yan1}], the author demonstrated the relation between (generalized) Drazin invertibility and one-sided (generalized) Drazin invertibility.

\begin{proposition} \label{1.4} \rm{([\cite{Yan1}])} Let $a\in\mathcal{A}$. Then

\rm{(i)} $a$ is both left and right Drazin invertible if and only if $a$ is Drazin invertible. Specifically, if $x$ is a left Drazin inverse of $a$ with index $j$, and $y$ is a right Drazin inverse of $a$ with index $k$, then $x=y$ and $k=j$.

\rm{(ii)} $a$ is both left and right generalized Drazin invertible if and only if $a$ is generalized Drazin invertible. Specifically, if $x$ is a left generalized Drazin inverse of $a$ and $y$ is a right generalized Drazin inverse of $a$, then $x=y$.
\end{proposition}

In [\cite{Yan1}], numerous equivalent characterizations of one-sided (generalized) Drazin invertibility are demonstrated in context of Banach space operators. For a Banach algebra $\mathcal{A}$, Jacobson\textquoteright s lemma for invertibility states that $1-ac$ is invertible if and only if $1-ca$ is invertible where $a,c\in \mathcal{A}$. This leads to the classical spectral identity
$$\sigma(ac)\backslash\{0\}=\sigma(ca)\backslash\{0\}.$$
The Jacobson\textquoteright s lemma for left and right invertibility can be found in [\cite{Muller}, Chapter I]. Over the last two decades, much research has been conducted on Jacobson\textquoteright s lemma for various types of generalized inverses, as see [\cite{Gonzalez}, \cite{Cvetkovic}, \cite{Lam}, \cite{Mosic}, \cite{Zhuang}]. Just as with invertibility, the natural question is whether Jacobson\textquoteright s lemma holds for one-sided generalized inverses. However, there is hardly any research on Jacobson\textquoteright s lemma for one-sided generalized inverses. One of the principal contributions of this paper is to establish the Jacobson\textquoteright s lemma for one-sided regularity and one-sided (generalized) Drazin inverse in Banach algebras.

Furthermore, Drazin [\cite{Drazin1}] introduced a new class of one-sided generalized inverse, namely one-sided $(b,c)$-inverse. Nevertheless, the question of whether Jacobson's lemma holds for one-sided (b,c)-inverse still remains open. To solve this problem, Drazin proposed seeking an analogue of the \textquotedblleft Azumaya realization\textquotedblright ~ for $(b,c)$-inverse, as discussed in [\cite{Drazin1}, Section 7]. Notably, it can be observed that our one-sided Drazin inverse is a subclass of the one-sided $(b,c)$-inverse when $b=c=a^j$. In this paper, the author establishes an analogue of the \textquotedblleft Azumaya realization\textquotedblright ~for one-sided Drazin inverse. Thus, examining Jacobson's lemma for one-sided Drazin inverse may provide a new approach to Drazin\textquoteright s open question on $(b,c)$-inverse in [\cite{Drazin1}, Section 7].

 This article is organized as follows. In Section 2, the author provides several fundamental properties, including intertwining properties and reverse order laws, of one-sided (generalized) Drazin inverse in Banach algebras. In particular, we introduce the concepts of one-sided strongly $\pi$-regularity, which is shown to be equivalent to one-sided Drazin invertibility and is pivotal in examining Jacobson\textquoteright s lemma for one-sided Drazin invertibility. In Section 3, Jacobson\textquoteright s lemma for one-sided regularity and one-sided (generalized) Drazin invertibility are established. As a corollary, the spectral identity of element products for one-sided (generalized) Drazin spectrum is obtained (Corollary \ref{3.11}). Surprisingly, the Cline formula is more difficult to prove than the Jacobson\textquoteright s lemma one-sided (generalized) Drazin invertibility. In Section 4, the spectral identity of two intertwining operators for one-sided (generalized) Drazin spectrum is established (Corollary \ref{4.7}). As applications, in Section 5, we utilize the results from Sections 3 and 4 to derive the spectral identities of some new types of operators, namely
$$\sigma_{*}(TS)\setminus\{0\}=\sigma_{*}(ST)\setminus\{0\},$$
where $T,S\in B(X)$ and $\sigma_{*}$ denotes the left and right B-Fredholm spectra. These spectral identities can be viewed as generalizations of classical conclusions regarding the spectra of one-sided Fredholm operators or one-sided essentially invertible operators (see [\cite{Muller}, p. 173, Chapter III, Theorem 6]).

\section{Fundamental properties of one-sided (generalized) Drazin inverses}

At the beginning of this section, we present the following properties, which are similar to those of one-sided Drazin invertibility as described in [\cite{Yan1}].

\begin{proposition} \label{3} Let $a\in \mathcal{A}$. Then

 \emph{(i)} $a$ is left generalized Drazin invertible if and only if there exists an element $b\in \mathcal{A}$ and an integer $j$ such that
\begin{equation}\label{eq1}
aba=ba^2,~~bab=b^2a=b \makebox{~~and~~}aba-a\in \mathcal{A}^{qnil}.
\end{equation}

\emph{(ii)} $a$ is right generalized Drazin invertible if and only if there exists an element $c\in \mathcal{A}$ and an integer $k$ such that
\begin{equation}\label{eq2}
aca=c^2a,~~cac=ac^2=c \makebox{~~and~~}aca-a\in \mathcal{A}^{qnil}.
\end{equation}
\end{proposition}

\begin{proof} Suppose that $a$ has a left generalized Drazin inverse $x$. Set $b=xax$, then
$$aba=a(xax)x^2a^3=(xax)a^2=ba^2,$$
$$bab=(xax)a(xax)=x^3a^3x=xax=b$$
and
$$b^2a=(xax)^2a=(xax)x^2a^2=xax^2a=xax=b.$$
Moreover, we have
$$aba-a=a(xax)a-a=axa-a\in \mathcal{A}^{qnil}.$$
Hence $b$ satisfies equation (\ref{eq1}). The computation of the right Drazin invertible case is similar.
\end{proof}

Intertwining properties are instrumental in computation of generalized inverses, see [\cite{Drazin1}]. Despite Ren and Jiang\textquoteright s exploration of the intertwining properties of one-sided Drazin inverses in [\cite{Ren}], we still exhibit following result (Proposition \ref{24} (i)) because our definitions of one-sided Drazin inverses enable us to observe \textquoteleft more\textquoteright~left or right Drazin inverses.

\begin{proposition} \label{24} Let $a, b\in \mathcal{A}$ and $az=zb$ for some $z\in \mathcal{A}$.

\emph{(i)} If $a$ has a left Drazin inverse $x$ and $b$ has a right Drazin invertible $y$, then $xz=zy$.

\emph{(ii)} If $a$ has a left generalized Drazin inverse $x$ and $b$ has a right generalized Drazin invertible $y$, then $xz=zy$.
\end{proposition}

\begin{proof} We need only to prove (ii), because the one-sided Drazin inverse of an element must be one-sided generalized Drazin inverse. Suppose that $a\in \mathcal{A}$ has a left generalized Drazin inverse $x$ and $b$ has a right generalized Drazin inverse $y$. For any non-negative integer $n$,
$$
\begin{aligned} ||xz-zy||^{\frac{1}{n}}&=||x^{n+1}a^nz-zb^ny^{n+1}||^{\frac{1}{n}}\\
&=||x^{n+1}a^nz-xazy+xazy-zb^ny^{n+1}||^{\frac{1}{n}}\\
&=||x^{n+1}zb^n-x^{n+1}a^{n+1}zy+xzb^{n+1}y^{n+1}-a^nzy^{n+1}||^{\frac{1}{n}}\\
&=||x^{n+1}zb^n-x^{n+1}zb^{n+1}y+xa^{n+1}zy^{n+1}-a^nzy^{n+1}||^{\frac{1}{n}}\\
&=||x^{n+1}z(b-byb)^n+(axa-a)^nzy^{n+1}||^{\frac{1}{n}}\\
&\leq ||x^{n+1}z(b-byb)^n||^{\frac{1}{n}}+||(axa-a)^nzy^{n+1}||^{\frac{1}{n}}\\
&\leq ||x||^{\frac{n+1}{n}}||z||^{\frac{1}{n+1}}||(a-aya)^n||^{\frac{1}{n}}+||(axa-a)^n||^{\frac{1}{n}}||z||^{\frac{1}{n+1}}||y||^{\frac{n+1}{n}}.
\end{aligned}
$$
Since $b-byb\in \mathcal{A}^{qnil}$ and $axa-a\in \mathcal{A}^{qnil}$, $\lim\limits_{n\rightarrow\infty}||xz-zy||^{\frac{1}{n}}=0$. Thus $||xz-zy||$ must be zero and so $xz=zy$.
\end{proof}

\begin{corollary} Let $a, b\in \mathcal{A}$ and $az=zb$ for some $z\in \mathcal{A}$. If $x$ is the Drazin inverse (resp. generalized Drazin inverse) of $a$ and $y$ is the Drazin inverse (resp. generalized Drazin inverse) of $b$, then $xz=zy$.
\end{corollary}

Let $a, b\in \mathcal{A}$ be both invertible, then $(ab)^{-1}=b^{-1}a^{-1}$. This was extended in [\cite{Koliha1}, Theorem 5.5] to generalized Drazin inverse under the condition $ab=ba$. Such a result was termed the reverse order law of generalized Drazin inverse. In the following, we aim to establish one-sided version of these results. Recall that the \emph{commutant} and \emph{double~commutant} of an element $a$ are defined by
$$\makebox{comm}(a)=\{x\in\mathcal{A}:~xa=ax\}$$
and
$$\makebox{comm}^2(a)=\{x\in \mathcal{A}:~xy=yx\makebox{~for~all~}y\in\makebox{comm}(a)\}.$$
It is well-known that if $x$ is the (generalized) Drazin inverse of $a\in\mathcal{A}$, then $x\in$ comm$^2(a)$.

\begin{theorem} Let $a, b \in \mathcal{A}$ and $a\in$ comm$^2(b)$.

 \emph{(i)} If $a$ has the Drazin inverse $x$ and $b$ has a left (resp. right) Drazin inverse $y$, then $yx$ is a left (resp. right) Drazin inverse of $ab$.

 \emph{(ii)} If $a$ has the generalized Drazin inverse $x$ and $b$ has a left (resp. right) generalized Drazin inverse $y$, then then $yx$ is a left (resp. right) generalized Drazin inverse of $ab$.
\end{theorem}

\begin{proof} We only prove the case of the left generalized Drazin inverse. Suppose that $x$ is the Drazin inverse of $a$ and $y$ is one of the left generalized Drazin inverses of $b$. Since $b(yb)=(yb)b$ and $a\in$ comm$^2(b)$, we have $a(yb)=(yb)a$. Furthermore, as $x\in$ comm$^2(a)$ and $ab=ba$, we can deduce that $xb=bx$. Therefore
$$\begin{aligned}
ab(yx)ab&=abybxa=ayb^2xa=ybabxa=yx(ab)(ab)
\end{aligned}
$$
and
$$(yx)^2(ab)=yxyxab=yxybxa=yybx^2a=yx.$$
Moreover,
$$\begin{aligned}
ab-yx(ab)^2=(a-a^2x)(b-yb^2)+a^2x(b-yb^2)+yb^2(a-a^2x).
\end{aligned}
$$
By commutativity, we have
$$\begin{aligned}
r(ab-yx(ab)^2)
\leq &r(a-a^2x)r(b-yb^2)+r(a^2x)r(b-yb^2)+\\
&r(yb^2)r(a-a^2x).
\end{aligned}
$$
Hence $r(ab-yx(ab)^2)=0$, which implies that $ab-yx(ab)^2$ is quasi-nilpotent.
\end{proof}

We are temporarily unable to extend the definition of the one-sided generalized Drazin inverse from Banach algebras to rings, as the analogue of many important properties (such as Proposition \ref{1.4}) of one-sided generalized Drazin inverse are unlikely to be established in the context of associative rings. The definition of one-sided generalized Drazin inverse in the normed algebra $C_0(X)=B(X)/F(X)$, which is not a Banach algebra, has been established in [\cite{AB1}]. Consequently, all the results concerning one-sided generalized inverse in this paper are valid and meaningful only in Banach algebras, and not in associative rings. It remains an interesting problem to find a suitable definition of one-sided generalized Drazin inverse in associative rings.

In a ring $\mathcal{R}$, an element $a\in \mathcal{R}$ is \textit{$\pi$-regular} if there exist an element $x$ and an integer $n$ such that $a^nxa^n=a^n$. The one-sided version of $\pi$-regular is called \textit{left $\pi$-regular} (resp. \textit{right $\pi$-regular}) if there exist an element $x$ and an integer $n$ such that $xa^{2n}=a^n$ (resp. $a^{2n}x=a^n$), see [\cite{Azu}, \cite{Kaplansky}]. It can be easily seen that $a\in \mathcal{R}$ is left $\pi$-regular (resp. right $\pi$-regular) if and only if there exists an element $x$ such that $xa^{n+1}=a^n$ (resp. $a^{n+1}x=a^n$). Then $a\in \mathcal{R}$ is strongly $\pi$-regular exactly when $a$ is both left and right $\pi$-regular. We now introduce the definitions of the one-sided analogue of Azumaya\textquoteright s strongly $\pi$-regular.

\begin{definition} \emph{An element $a\in \mathcal{R}$ is called \textit{left strongly $\pi$-regular}, if there exist an element $x\in \mathcal{R}$ and an integer $p$ such that
$$ axa=xa^2,~~~ xa^{p+1}=a^p.$$
If such $x$ exists, it is called \textit{left inner inverse} of $a$. Dually, an element $a\in \mathcal{A}$ is called \textit{right strongly $\pi$-regular}, if there exist an element $y\in \mathcal{A}$ and an integer $q$ such that
$$aya=a^2y,~~~ a^{q+1}y=a^q.$$
If such $y$ exists, it is called \textit{right inner inverse} of $a$.}
\end{definition}

The so called \textquotedblleft Azumaya realization\textquotedblright~was obtained by Drazin in [\cite{Drazin}, Theorem 4]. The \textquotedblleft Azumaya realization\textquotedblright~ of bounded linear operators in the context of Banach spaces is discussed in [\cite{Muller}, Chapter III]. Subsequently, the following result demonstrates the one-sided version of classical \textquotedblleft Azumaya realization\textquotedblright~ for Drazin inverse.

\begin{theorem} \label{1} Let $a\in \mathcal{R}$. Then

 \emph{(i)} $a$ is left Drazin invertible if and only if $a$ is left strongly $\pi$-regular.

\emph{(ii)} $a$ is right Drazin invertible if and only if $a$ is right strongly $\pi$-regular.
\end{theorem}

\begin{proof} (i) If $a$ is left Drazin invertible with index $j$, then it is obvious that $a$ is left strongly $\pi$-regular. Conversely, suppose that $a$ is left strongly $\pi$-regular, i.e. there exist an element $x$ and an integer $p$ such that  $axa=xa^2$ and $xa^{p+1}=a^p$. Put $c=x^{p+1}a^p$, then
$$aca=ax^{p+1}a^pa=x^{p+1}a^{p+2}=(x^{p+1}a^p)a^2=ca^2.
$$
and
$$\begin{aligned}
c^2a&=x^{p+1}a^px^{p+1}a^pa=x^{2p+2}a^{2p+1}=x^{2p+1}(xa^{p+1})a^p\\
&=x^{2p+1}a^{2p}=...=x^{p+1}a^p=c.
\end{aligned}
$$
Moreover
$$ca^{p+1}=x^{p+1}a^pa^{p+1}=x^{p+1}a^{2p+1}=x^pa^{2p}=...=a^p.$$
Hence $c$ is the left Drazin inverse of $a$ with the index $p$.

(ii) The proof is dual to that of (i).
\end{proof}

It is evident that if an element $a\in \mathcal{R}$ is both left and right strongly $\pi$-regular, then $a$ is strongly $\pi$-regular. By Theorems \ref{6} and \ref{1}, we can immediately obtain the (1) of Theorem \ref{1.4} without any computation.

\begin{remark} \label{rk1} \emph{Let $a\in \mathcal{R}$ be left strongly $\pi$-regular, i.e. there is an element $x\in \mathcal{R}$ such that $axa=xa^2$, $xa^{p+1}=a^p$. From the proof of Theorem \ref{1}, we observe that $x^{p+1}a^p$ is a left Drazin inverse of $a$ with the index $p$. If $a\in \mathcal{A}$ be right strongly $\pi$-regular with corresponding right regular inverse $y\in \mathcal{R}$ and index $q$, then the right Drazin inverses of $a$ is $a^qy^{q+1}$ with Drazin index $q$. This observation plays an important role in the next section.}
\end{remark}

Recall that a ring $\mathcal{R}$ is called \textit{regular} (or Von Neumann regular) if every element in $\mathcal{R}$ is regular. The concept of a regular ring, first introduced in von Neumann\textquoteright s pioneering work, continues to serve as a valuable tool in the study of operator algebras. As a generalization of regular ring and algebraic algebra, Kaplansky firstly introduced and studied the $\pi$-regular rings in [\cite{Kaplansky}]. If every element $a\in \mathcal{R}$ is $\pi$-regular (resp. left $\pi$-regular, right $\pi$-regular and strongly $\pi$-regular), then the ring $\mathcal{R}$ itself is called \textit{$\pi$-regular} (resp. \textit{left $\pi$-regular, right $\pi$-regular and strongly $\pi$-regular}). A ring is called \textit{bounded index} if there is a fixed upper bounded to the indices of nilpotent elements.

\begin{corollary}
If $\mathcal{R}$ is a ring of bounded index, then the following conditions are equivalent:

 \emph{(i)} $\mathcal{R}$ is $\pi$-regular;

 \emph{(ii)} $\mathcal{R}$ is right $\pi$-regular;

 \emph{(iii)} $\mathcal{R}$ is left $\pi$-regular;

 \emph{(iv)} $\mathcal{R}$ is strongly $\pi$-regular;

 \emph{(v)} $\mathcal{R}$ is right strongly $\pi$-regular;

 \emph{(vi)} $\mathcal{R}$ is left strongly $\pi$-regular.

\end{corollary}

\begin{proof} It follows directly from [\cite{Azu}, Theorem 5] and the fundamental relations among these regularities.
\end{proof}

\section{Jacobson\textquoteright s lemma for the one-sided regularity and one-sided Drazin invertibility}
\label{sec3}

Throughout this section, we will investigate the Jacobson\textquoteright s lemma for one-sided regularity and one-sided Drazin invertibility. Let $\mathcal{A}$ be a unital Banach algebra and $a, b, c, d\in \mathcal{A}$. To simplify the expressions of formulas, we introduce the notation $\alpha:=1-ac$ and $\beta:=1-bd$, which satisfy the equalities
\begin{gather}\label{eq3.1}
acd=dbd\makebox{~~and~~}dba=aca. \tag{3.1}
\end{gather}
The equalities (\ref{eq3.1}) were first introduced by the author and his collaborator in [\cite{Yan3}], and they encompass several interesting special cases:

 (I) When $a=d$ and $b=c$, the equalities (\ref{eq3.1}) always hold. Thus, $\alpha$ and $\beta$ correspond to the original Jacobson\textquoteright s lemma.

 (II) When $a=d$ and $b=1$, the equalities (\ref{eq3.1}) become $aba=a^2$, which is of interest to operator theorists; see [\cite{Aiena}] for instance.

 (III) When $a=d$, the equalities (\ref{eq3.1}) become $aba=aca$, which was systematically studied in [\cite{Corach}, \cite{Zeng}].

\noindent Hence, we consider the one-sided regularity and one-sided Drazin invertibility of $\alpha$ and $\beta$ under conditions (\ref{eq3.1}) in this section.

Now we recall some definitions of regularities in Banach algebras or rings. An element $a\in \mathcal{A}$ is called \emph{regular} (or \emph{Von Neumann regular}) if there exists an element $x$ such that $axa=a$, while $a\in \mathcal{A}$ is termed \emph{left regular} (resp. \emph{right regular}) if there exists an element $x$ such that $xa^2=a$ (resp. $a^2x=a$), see Azumaya [\cite{Azu}] and Kaplansky [\cite{Kaplansky}]. Recently, an extension of Jacobson\textquoteright s lemma for Von Neumann regularity was derived in [\cite{Zeng}], providing an affirmative answer to a question posed by Corach, Duggal, and Harte in [\cite{Corach}]. In the following, we give a one-sided analogue of this result.

\begin{theorem}\label{4} Let $\alpha, \beta\in \mathcal{A}$ satisfy \emph{(\ref{eq3.1})}. Then $\alpha$ is left (resp. right) regular if and only if $\beta$ is left (resp. right) regular.
\end{theorem}

\begin{proof} Suppose that $1-ac$ is left regular. Then there exists an element $x$ such that $x(1-ac)^2=1-ac$. Put $y=1+bd+bacxd$. Thus
$$
\begin{aligned} bacxd(1-bd)(1-bd)&=bacx(1-ac)d(1-bd)=bacx(1-ac)^2d\\
&=bac(1-ac)d=bd(1-bd)bd.
\end{aligned}
$$
Hence
$$
\begin{aligned} y(1-bd)^2&=(1+bd+bacxd)(1-bd)^2\\
&=(1+bd)(1-bd)^2+bacxd(1-bd)^2\\
&=1-bd-(bd)^2+(bd)^3+bd(1-bd)bd\\
&=1-bd,
\end{aligned}
$$
 this implies that $1-bd$ is also left regular. Conversely, suppose that $1-bd$ is left regular with such inverse $y$. Then one can similarly check that $(1+ac+dybac)(1-ac)^2=1-ac$, this implies $1-ac$ is also left regular. The proof of right regularity is dual to that of left regularity.
\end{proof}

By taking $a=d$ in Theorem \ref{4}, we obtain the one-sided version of [\cite{Zeng}, Theorem 2.1]. Additionally, if we take $a=d$ and $b=c$ in Theorem \ref{4}, we obtain following corollary.

\begin{corollary} If $a, c\in \mathcal{A}$. Then $1-ac$ is left regular (resp. right regular) if and only if $1-ca$ is left regular (resp. right regular).
\end{corollary}

\begin{theorem}\label{5} Let $\alpha, \beta\in \mathcal{A}$ satisfy \emph{(\ref{eq3.1})}. Then

 \emph{(i)} $\alpha$ is left $\pi$-regular (resp. right $\pi$-regular) if and only if $\beta$ is left $\pi$-regular (resp. right $\pi$-regular).

 \emph{(ii)} $\alpha$ is left strongly $\pi$-regular (resp. right strongly $\pi$-regular) if and only if $\beta$ is left strongly $\pi$-regular (resp. right strongly $\pi$-regular).
\end{theorem}

\begin{proof} (i) Set
$$b_n=\sum\limits_{i=1}^n\Big(\normalsize\begin{array}{ccc}
    n \\
    i \\
  \end{array}\Big)(-1)^i(bd)^{i-1}b~~~\makebox{and}~~~c_n=\sum\limits_{i=1}^i\Big(\normalsize\begin{array}{ccc}
    n \\
    i \\
  \end{array}\Big)(-1)^ic(ac)^{i-1}.$$
  Then $ac_nd=db_nd$ and $ac_na=db_na$. From Theorem \ref{4}, it follows that $(1-bd)^n=1-b_nd$ is left regular (resp. right regular) if and only if $(1-ac)^n=1-ac_n$ is left regular (resp. right regular). Thus the conclusion holds.

 (ii) Suppose that $1-ac$ is left strongly $\pi$-regular. Then there exist an integer $p$ and an element $x\in \mathcal{A}$ such that $(1-ac)x(1-ac)=x(1-ac)^2$ and $x(1-ac)^{p+1}=(1-ac)^p$. Set $y=1+bd+bacxd$. First, from
 $$
\begin{aligned} bacxd(1-bd)^2&=bacx(1-ac)^2d=bac(1-ac)d=bd(1-bd)bd
\end{aligned}
$$
and
$$
\begin{aligned} (1-bd)bacxd(1-bd)&=bac(1-ac)x(1-ac)d=bac(1-ac)d=bd(1-bd)bd,
\end{aligned}
$$
it follows that $bacxd(1-bd)^2=(1-bd)bacxd(1-bd)$. Then
$$
\begin{aligned} (1-bd)y(1-bd)&=(1-bd)(1+bd+bacxd)(1-bd)\\
&=(1-bd)(1+bd)(1-bd)+(1-bd)bacxd(1-bd)\\
&=(1+bd)(1-bd)^2+bacxd(1-bd)^2\\
&=(1+bd+bacxd)(1-bd^2)\\
&=y(1-bd)^2.
\end{aligned}
$$
Next, since
$$
\begin{aligned} bacxd(1-bd)^{p+1}&=bacx(1-ac)^{p+1}d=bac(1-ac)^pd\\
&=bacd(1-bd)^p=bdbd(1-bd)^p,
\end{aligned}
$$
we have
$$
\begin{aligned} y(1-bd)^{p+1}&=(1+bd+bacxd)(1-bd)^{p+1}\\
&=(1+bd)(1-bd)^{p+1}+bacxd(1-bd)^{p+1}\\
&=(1+bd)(1-bd)^{p+1}+bdbd(1-bd)^p\\
&=[(1+bd)(1-bd)+bdbd](1-bd)^p\\
&=(1-bd)^p.
\end{aligned}
$$
Hence $1-bd$ is also left strongly $\pi$-regular. Conversely, suppose that $1-bd$ is also left strongly $\pi$-regular, i.e. there exist an integer $p$ and an element $y\in \mathcal{A}$ such that $(1-bd)y(1-bd)=x(1-bd)^2$ and $x(1-bd)^{p+1}=(1-bd)^p$. The proof is analogous when you notice the relevant $x=1+ac+dybac$ corresponding to $1-ac$.
\end{proof}

Taking a special case $a=d$ and $b=c$ in Theorem \ref{5}, we obtain following result, which contains the one-sided version of [\cite{Lam}, Theorem 3.8].

\begin{corollary} Let $a, c\in \mathcal{A}$. Then $1-ac$ is left $\pi$-regular (resp. right $\pi$-regular, left strongly $\pi$-regular and right strongly $\pi$-regular) if and only if $1-ca$ is left $\pi$-regular (resp. right $\pi$-regular, left strongly $\pi$-regular and right strongly $\pi$-regular).
\end{corollary}

The next two results illustrate the one-sided analogues of the formulas in [\cite{Yan1}, Theorem 3.1]. It should be noted that the Jacobson\textquoteright s lemma for one-sided Drazin inverse has been obtained by Ren and Jiang in [\cite{Ren}]. However, we utilize our one-sided analogue of the \textquotedblleft Azumaya realization\textquotedblright~ to simplify the proof of this result.

\begin{theorem} \label{3.5} Let $\alpha, \beta\in \mathcal{A}$ satisfy \emph{(\ref{eq3.1})}. $\alpha$ is left (resp. right) Drazin invertible with index $k$ if and only if $\beta$ is left (resp. right) Drazin invertible with index $k$. More precisely, if $\alpha$ has a left (resp. right) Drazin inverse $x$, then
$$y=(1-bacprd)(1+bd)+bacxd$$
is a left (resp. right) Drazin inverse of $\beta$. Conversely, if $\beta$ has a left (resp. right) Drazin inverse $y$, then
$$x=(1-dp'r'bac)(1+ac)+dybac$$
is a left (resp. right) Drazin inverse of $\alpha$. In both two cases, $r=\sum_{j=0}^{k-1}[1-(ac)^{2}]^{j}$, $r'=\sum_{j=0}^{k-1}[1-(bd)^{2}]^{j}$, $p=1-x\alpha$ and $p'=1-y\beta$ (resp. $p=1-\alpha x$ and $p'=1-\beta y$).
\end{theorem}

\begin{proof} We only present the proof for left Drazin invertibility, as the proof for right Drazin invertibility is dual. Suppose that $x$ is a left Drazin inverse of $\alpha$. Then $x$ is also a left inner inverse of $\alpha$. From the proof of Theorem \ref{5}, it follows that $1+bd+bacxd$ is a left inner inverse of $\beta$. By Remark \ref{rk1}, $[1+bd+bacxd]^{k+1}\beta^k$ is a left Drazin inverse of $\beta$. Set $y:=(1+bd+bacx'd)^{k+1}\beta^k$ and $p:=1-x\alpha$. Since $\alpha x\alpha=x(\alpha)^2,$
we have $(ac)(x\alpha)=(x\alpha)(ac)$, which implies $(ac)p=p(ac)$. Hence
$$bacxd(1-bacpd)=bacxd-bacxdbacpd=bacxd-bacxp(ac)^2d=bacxd.$$
Also $(1+bd+bacx'd)\beta=1-bacpd$, and so
$$
\begin{aligned} y&=(1+bd+bacxd)^{k+1}\beta^k=(1+bd+bacxd)(1-bacpd)^k\\
&=(1+bd)(1-bacpd)^k+bacxd(1-bacpd)^k\\
&=(1-bacpd)^k(1+bd)+bacxd\\
&=(1-bacprd)(1+bd)+bacxd,
\end{aligned}
$$
where $r=\sum_{j=0}^{k-1}[1-(ac)^{2}]^{j}$. Conversely, one can get the conclusion by similar proof and computation as above.
\end{proof}

By taking the one-sided Drazin index $k=1$, we immediately get the following result about one-sided group inverse.

\begin{theorem} Let $\alpha, \beta\in \mathcal{A}$ satisfy \emph{(\ref{eq3.1})}. Then $\alpha$ is left (resp. right) group invertible if and only if $\beta$ is left (resp. right) group invertible. More precisely, if $\alpha$ has a left (resp. right) group inverse $x$, then
$$y=(1-bacpd)(1+bd)+bacxd$$
is a left (resp. right) group inverse of $\beta$. Conversely, if $\beta$ has a left (resp. right) inverse $y$, then
$$x=(1-dp'bac)(1+ac)+dybac$$
is a left (resp. right) inverse of $\alpha$. In both two cases, $p=1-x\alpha$ and $p'=1-y\beta$ (resp. $p=1-\alpha x$ and $p'=1-\beta y$).
\end{theorem}

We now explore Jacobson\textquoteright s lemma for one-sided generalized Drazin inverses in Banach algebras. The reason for considering this problem within the realm of Banach algebras is that the question of whether there exists a suitable definition for one-sided generalized Drazin inverses in rings remains unresolved, see Remark \ref{r1}.

\begin{theorem}\label{3.6} Let $\alpha,~\beta\in \mathcal{A}$ satisfy \emph{(\ref{eq3.1})}. Then $\alpha$ is left (resp. right) generalized Drazin invertible if and only if $\beta$ is left (resp. right) generalized Drazin invertible. More precisely, if $\alpha$ has a left (resp. right) generalized Drazin inverse $x$, then
$$y=(1-bacp[1-p\alpha(1+ac)]^{-1}d)(1+bd)+bacxd$$
is a left (resp. right) generalized Drazin inverse of $\beta$. Conversely, if $\beta$ has a left (resp. right) generalized Drazin inverse $y$, then
$$x=(1-dp'[1-p'\beta(1+ac)]^{-1}bac)(1+ac)+dybac$$
is a left (resp. right) generalized Drazin inverse of $\alpha$. In both two cases, $p=1-x\alpha$ and $p'=1-y\beta$ (resp. $p=1-\alpha x$ and $p'=1-\beta y$).
\end{theorem}

\begin{proof} We only present the proof for left generalized Drazin invertibility, as the proof for right generalized Drazin invertibility is dual. In order to prove that $y$ is a left generalized Drazin inverse of $\beta$, it suffices to show that the following three conditions hold: (i) $\beta y\beta=y\beta^2$; (ii) $y^2\beta=y$; (iii) $\beta-\beta y\beta$ is quasi-nilpotent.

(i) Since $p$ and $ac$ commute, $1-p\alpha(1+ac)$ is invertible. Thus
$$
\begin{aligned}   y\beta&=1-(bd)^{2}-bacp[1-p\alpha(1+ac)]^{-1}d(1+bd)(1-bd)+bacxd(1-bd)\\
&=1-[bacd-bacx\alpha d]-bacp[1-p\alpha(1+ac)]^{-1}\alpha(1+ac)d\\
&=1-bacpd-bacp[1-p\alpha(1+ac)]^{-1}\alpha(1+ac)d\\
&=1-bacp\{1+[1-p\alpha(1+ac)]^{-1}p\alpha(1+ac)\}d\\
&=1-bacp[1-p\alpha(1+ac)]^{-1}d
\end{aligned}
$$
and so
$$
\begin{aligned} \beta y\beta&=\beta-\beta bacp[1-p\alpha(1+ac)]^{-1}d=\beta-bac\alpha p[1-p\alpha(1+ac)]^{-1}d\\
&=\beta-bac[1-p\alpha(1+ac)]^{-1}\alpha d=\beta-bac[1-p\alpha(1+ac)]^{-1}d\beta=y\beta^2.
\end{aligned}
$$

(ii) From the commutativity of $ac$ and $p$, it follows that
$$
\begin{aligned} (bacxd)bacp[1-p\alpha(1+ac)]^{-1}d=bac(xp)(ac)^2[1-p\alpha(1+ac)]^{-1}d.\\
\end{aligned}
$$
Since $x$ is a left generalized Drazin inverse of $\alpha$, $xp=0$ and so
$$(bacxd)bacp[1-p\alpha(1+ac)]^{-1}d=0.$$
Hence, by the commutativity of $ac$ and $p$ again,
$$
\begin{aligned} y^2\beta&=y-ybacp[1-p\alpha(1+ac)]^{-1}d\\
&=y-bac(1+ac)p[1-p\alpha(1+ac)]^{-1}d+bac(1+ac)p[1-p\alpha(1+ac)]^{-2}(ac)^2d\\
&=y-bac(1+ac)p[1-p\alpha(1+ac)]^{-2}[p-p\alpha(1+ac)-p(ac)^2]d\\
&=y.
\end{aligned}
$$

(iii) From the proof of part (i), we have
$$\beta-\beta y\beta=bac\alpha p[1-p\alpha(1+ac)]^{-1}d.$$
Therefore
$$
\begin{aligned} r(\beta-\beta y\beta)&=r(bac\alpha p[1-p\alpha(1+ac)]^{-1}d)\leq r([1-p\alpha(1+ac)]^{-1}dbac\alpha p)\\
&=r([1-p\alpha(1+ac)]^{-1}acac\alpha p)\leq r([1-p\alpha(1+ac)]^{-1}acac)r(\alpha p)\\
&=r([1-p\alpha(1+ac)]^{-1}acac)r(\alpha-\alpha x\alpha)=0,
\end{aligned}
$$
which implies that $\beta-\beta y\beta$ is quasi-nilpotent. The proof of the opposite implication is parallel to that above.
\end{proof}

\begin{corollary} \label{310} Let $a, c\in \mathcal{A}$. Then

\emph{(i)} $1-ac$ is left (resp. right) Drazin invertible with index $k$ if and only if $1-ca$ is left (resp. right) Drazin invertible with index $k$.

\emph{(ii)} $1-ac$ is left (resp. right) group invertible if and only if $1-ca$ is left (resp. right) group invertible.

\emph{(iii)} $1-ac$ is left (resp. right) generalized Drazin invertible if and only if $1-ca$ is left (resp. right) generalized Drazin invertible.
\end{corollary}

At the end of this section, we provide the spectral identity of one-side (generalized) Drazin inverse in Banach algebras. Now we define
$$\sigma_{ld}(a):=\{\lambda\in \mathbb{C}: \lambda-a \makebox{~is~not~left~Drazin~invertible~in}~\mathcal{A}\},$$
$$\sigma_{rd}(a):=\{\lambda\in \mathbb{C}: \lambda-a \makebox{~is~not~right~Drazin~invertible~in}~\mathcal{A}\},$$
$$\sigma_{lg}(a):=\{\lambda\in \mathbb{C}: \lambda-a \makebox{~is~not~left~group~invertible~in}~\mathcal{A}\},$$
$$\sigma_{rg}(a):=\{\lambda\in \mathbb{C}: \lambda-a \makebox{~is~not~right~group~invertible~in}~\mathcal{A}\},$$
$$\sigma_{lgd}(a):=\{\lambda\in \mathbb{C}: \lambda-a \makebox{~is~not~left~generalized~Drazin~invertible~in}~\mathcal{A}\},$$
and
$$\sigma_{rgd}(a):=\{\lambda\in \mathbb{C}: \lambda-a \makebox{~is~not~left~generalized~Drazin~invertible~in}~\mathcal{A}\}.$$
By corollary \ref{310}, we immediately obtain the following spectral identity for one-sided (generalized) Drazin spectrum.

\begin{corollary} \label{3.11} Let $a, c\in \mathcal{A}$. Then
$$\sigma_*(ac)\setminus \{0\}=\sigma_*(ca)\setminus \{0\},$$
where $*$=$\{$ld, rd, lg, rg, lgd, rgd$\}$.
\end{corollary}

It is also interesting to consider whether the Cline formula holds for one-sided (generalized) Drazin inverses in Banach algebras, namely the relation between the zero points of the one-sided (generalized) Drazin spectrum of element products $ac$ and $ca$ in Banach algebras. However, we conjecture that Cline formula may not hold for one-sided (generalized) Drazin inverses.

\begin{question} Does the Cline formula hold for one-sided (resp. generalized) Drazin inverses in rings (resp. Banach algebras)?
\end{question}

We now  present Cline formula for one-sided Drazin inverses under certain additional conditions. A comparable result concerning one-sided generalized Drazin inverses can be derived through analogous computations.

\begin{proposition} Let $a, c\in \mathcal{A}$

\emph{(i)} If $ac$ is left Drazin invertible with index $k$ and $c$ is right invertible, then $ca$ is left Drazin invertible with index $k+1$;

\emph{(ii)} If $ac$ is right Drazin invertible with index $k$ and $a$ is left invertible, then $ca$ is right Drazin invertible with index $k+1$;
\end{proposition}

\begin{proof} Suppose that $ac$ is left Drazin invertible and $x$ is a left Drazin inverse of $ac$ with left Drazin index $k$. Set $y=cx^2a$. To prove that $y$ is a left Drazin inverse of $ca$ with index $k+1$, it suffices to demonstrate that the following three conditions hold: $(ca)y(ca) = y(ca)^2$, $y^2(ca)=y$ and $y(ca)^{k+2}=(ca)^{k+1}$.

First, since
$$(ca)y(ca)c=cacx^2(ac)^2=cacxacxac=cx^2(ac)^3=y(ca)^2c$$
and $c$ is right invertible, we have $(ca)y(ca)=y(ca)^2$.
Next,
$$y^2(ca)=cx^2acx^2a(ca)=cx^2a=y.$$
Finally, we have
$$y(ca)^{k+2}=cx^2a(ca)^{k+2}=cx(ac)^{k+1}a=c(ac)^{k}a=(ca)^{k+1}.$$
Consequently, $ca$ is Drazin invertible with index $k+1$.
\end{proof}

\section{Common one-sided Drazin invertibility of two elements under certain intertwining conditions}
\label{sec4}

Let $a,~b$ be two elements in $\mathcal{A}$ satisfying intertwining conditions
\begin{gather}\label{eq4.1}
ab^n=b^{n+1} \makebox{~~and~~} ba^n=a^{n+1} \tag{4.1}
\end{gather}
for some positive integer $n$. In this section, we primarily focus on the common one-sided (generalized) Drazin invertibility of $1-a$ and $1-b$ under intertwining conditions (\ref{eq4.1}). The formulas for (generalized) Drazin inverses of $1-a$ and $1-b$ are also provided.

\begin{theorem} \label{4.0} Let $a, b\in \mathcal{A}$ satisfy \emph{(\ref{eq4.1})}. Then $1-a$ is left (resp. right) regular if and only if $1-b$ is left (resp. right) regular.
\end{theorem}

\begin{proof} Suppose that $1-a$ is left regular. Then there exists an element $x$ such that $x(1-a)^2=1-a$. Set
$$y=1+\sum\limits_{i=1}^{2n-1}b^i+a^nxb^n.$$
From the definition of left regular, it is routine to examine that $y(1-b)^2=1-b$. This implies that $1-b$ is also left regular. Conversely, let $1-b$ is left regular with correspond inverse $y$. Then it is routine to examine that
$$(1+\sum\limits_{i=1}^{2n-1}a^i+b^nya^n)(1-a)^2=1-a.$$
Hence $1-a$ is also left regular. The proof of right regularity is similar to that of left regularity.
\end{proof}

\begin{theorem} \label{4.1} Let $a, b\in \mathcal{A}$ satisfy \emph{(\ref{eq4.1})}. Then $1-a$ is left (resp. right) strongly $\pi$-regular if and only if $1-b$ is left (resp. right) strongly $\pi$-regular.
\end{theorem}

\begin{proof} We only prove the case of left strongly $\pi$-regular. Suppose that $1-a$ is left strongly $\pi$-regular, namely there exists an element $x\in \mathcal{A}$ and an integer $p$ such that $(1-a)x(1-a)=x(1-a)^2$ and $x(1-a)^{p+1}=(1-a)^p$. Set
$$y=1+\sum\limits_{i=1}^{2n-1}b^i+a^nxb^n.$$
First, we have
$$
\begin{aligned} (1-b)y(1-b)&=(1-b)(1+\sum\limits_{i=1}^{2n-1}b^i+a^nxb^n)(1-b)\\
&=(1-b)^2+b-b^2-b^{2n}+b^{2n+1}+a^n(1-a)x(1-a)b^n\\
&=1-b-b^{2n}+b^{2n+1}+a^nx(1-a)^2b^n\\
\end{aligned}
$$
and
$$\begin{aligned} y(1-b)^2&=(1+\sum\limits_{i=1}^{2n-1}b^i+a^nxb^n)(1-b)^2=1-b-b^{2n}+b^{2n+1}+a^nx(1-a)^2b^n.
\end{aligned}
$$
Thus $(1-b)y(1-b)=y(1-b)^2$. Next
$$
\begin{aligned} y(1-b)^{p+1}&=(1+\sum\limits_{i=1}^{2n-1}b^i+a^nxb^n)(1-b)^{p+1}\\
&=(1+\sum\limits_{i=1}^{2n-1}b^i)(1-b)^{p+1}+a^nxb^n(1-b)^{p+1}\\
&=(1+\sum\limits_{i=1}^{2n-1}b^i)(1-b)^{p+1}+a^nx(1-a)^{p+1}b^n\\
&=(1+\sum\limits_{i=1}^{2n-1}b^i)(1-b)^{p+1}+a^n(1-a)^{p}b^n\\
&=(1+\sum\limits_{i=1}^{2n-1}b^i)(1-b)^{p+1}+a^nb^n(1-b)^{p}\\
&=[(1+\sum\limits_{i=1}^{2n-1}b^i)(1-b)+a^nb^n](1-b)^p\\
&=(1-b)^p.
\end{aligned}
$$
Consequently, $1-b$ is left strongly $\pi$-regular. Conversely, one can get the conclusion by similar proof and computation as above.
\end{proof}

Applying similar computational techniques as those presented in Section 3, we obtain the formulas for one-sided (generalized) Drazin inverses of $1-a$ and $1-b$ under the conditions (\ref{eq4.1}).

\begin{theorem} \label{4.2} Let $a, b\in \mathcal{A}$ satisfy \emph{(\ref{eq4.1})}. Then $1-a$ is left (resp. right) Drazin invertible with index $k$ if and only if $1-b$ is left (resp. right) Drazin invertible with index $k$. More precisely, if $1-a$ has a left (resp. right) Drazin inverse $x$ with index $k$, then
$$y=(1-a^nrpb^n)(1+\sum\limits_{i=1}^{2n-1}b^i)+a^nxb^n$$
is a left (resp. right) Drazin inverse of $1-b$. Conversely, if $1-b$ has a left (resp. right) Drazin inverse $y$ with index $k$, then
$$x=(1-b^np'r'a^n)(1+\sum\limits_{i=1}^{2n-1}a^i)+b^nxa^n$$
is a left (resp. right) Drazin inverse of $1-a$. In both two case, $r=\sum\limits_{i=0}^{k-1}(1-a^{2n})^i$, $r'=\sum\limits_{i=0}^{k-1}(1-b^{2n})^i$, $p=1-x(1-a)$ and $p'=1-y(1-b)$ (resp. $p=1-(1-a)x$ and $p'=1-(1-b)y$).
\end{theorem}

\begin{proof} We only present the proof for left Drazin invertibility, as the proof for right Drazin invertibility is dual. Suppose that $1-a$ has a left Drazin inverse $x$. Then $x$ is also a left inner inverse of $1-a$. By Theorem \ref{4.1}, $1-b$ is left (resp. right) strongly $\pi$-regular with a left inner inverse $1+\sum\limits_{i=1}^{2n-1}b^i+a^nxb^n$. From the Remark \ref{rk1}, $y=(1+\sum\limits_{i=1}^{2n-1}b^i+a^nxb^n)^{k+1}(1-b)^k$ is a left Drazin inverse of $1-b$. Since
$$(1+\sum\limits_{i=1}^{2n-1}b^i+a^nxb^n)(1-b)=1-b^{2n}+a^nx(1-a)b^n$$
and $1-b$ commutes with $a^nx(1-a)b^n$, we have
$$y=(1+\sum\limits_{i=1}^{2n-1}b^i+a^nxb^n)[1-b^{2n}+a^nx(1-a)b^n]^k.$$
It should be noted that $a$ commutes with $x(1-a)$. Then
$$
\begin{aligned}&a^nxb^n[1-b^{2n}+a^nx(1-a)b^n]=a^nxb^n-a^nxb^{3n}+a^nxb^na^nx(1-a)b^n\\
&=a^nxb^n-a^nxa^{2n}b^n+a^nxa^{2n}x(1-a)b^n=a^nxb^n-a^nxa^{2n}[1-x(1-a)]b^n\\
&=a^nxb^n-a^{3n}x[1-x(1-a)]b^n=a^nxb^n.
\end{aligned}
$$
Hence
$$
\begin{aligned}y&=(1+\sum\limits_{i=1}^{2n-1}b^i)[1-b^{2n}+a^nx(1-a)b^n]^k+a^nxb^n[1-b^{2n}+a^nx(1-a)b^n]^k\\
&=[1-b^{2n}+a^nx(1-a)b^n]^k(1+\sum\limits_{i=1}^{2n-1}b^i)+a^nxb^n\\
&=(1-a^npb^n)^k(1+\sum\limits_{i=1}^{2n-1}b^i)+a^nxb^n\\
&=(1-a^nprb^n)(1+\sum\limits_{i=1}^{2n-1}b^i)+a^nxb^n,\\
\end{aligned}
$$
where $r=\sum\limits_{i=0}^{k-1}(1-a^{2n})^i$. This completes the proof. Conversely, one can get the conclusion by similar proof and computation as above.
\end{proof}

By taking the one-sided Drazin index $k=1$, we obtain the Jacobson\textquoteright s lemma for one-sided group inverse.

\begin{theorem} \label{4.3} Let $a, b\in \mathcal{A}$ satisfy \emph{(\ref{eq4.1})}. Then $1-a$ is left (resp. right) group inverse if and only if $1-b$ is left (resp. right) group inverse. More precisely, if $1-a$ has a left (resp. right) group inverse, then
$$y=(1-a^npb^n)(1+\sum\limits_{i=1}^{2n-1}b^i)+a^nxb^n$$
is a left (resp. right) group inverse of $1-b$. Conversely, if $1-b$ has a left (resp. right) group inverse $y$, then
$$x=(1-b^np'a^n)(1+\sum\limits_{i=1}^{2n-1}a^i)+b^nxa^n$$
is a left (resp. right) group inverse of $1-a$. In both two cases, $p=x(1-a)$ and $p'=1-y(1-b)$ (resp. $p=1-(1-a)x$ and $p'=1-(1-b)y$).
\end{theorem}

\begin{theorem} \label{4.5} Let $a, b\in \mathcal{A}$ satisfy $ab^n=b^{n+1}$ and $ba^n=a^{n+1}$. Then $1-a$ is left (resp. right) generalized Drazin invertible if and only if $1-b$ is left (resp. right) generalized Drazin invertible. More precisely, if $1-a$ has a left (resp. right) generalized Drazin inverse $x$, then
$$y=(1-a^n[1-p(1-a^{2n})]^{-1}pb^n)(1+\sum\limits_{i=1}^{2n-1}b^i)+a^nxb^n$$
is a left (resp. right) generalized Drazin inverse of $1-b$. Conversely, if $1-b$ has a left $($resp. right$)$ generalized Drazin inverse $y$, then
$$x=(1-b^np'[1-p'(1-b^{2n})]^{-1}a^n)(1+\sum\limits_{i=1}^{2n-1}a^i)+b^nxa^n$$
is a left (resp. right) generalized Drazin inverse of $1-a$. In both two cases, $p=1-x(1-a)$ and $p'=1-y(1-b)$ (resp. $p=1-(1-a)x$ and $p'=1-(1-b)y$).
\end{theorem}

\begin{proof} We only present the proof for left generalized Drazin invertibility, as the proof for right generalized Drazin invertibility is dual. In order to prove that $y$ is a left generalized Drazin inverse of $1-b$, it suffices to show that the following three conditions hold : (i) $(1-b)y(1-b)=y(1-b)^2$; (ii) $y^2(1-b)=y$; (iii) $(1-b)y(1-b)-(1-b)$ is quasi-nilpotent.

(i) Since $p$ commutes with $a$,
$$1-p(1-a^{2n})=1-p(1+\sum\limits_{i=1}^{2n-1}a^i)(1-a)$$
is invertible. Then
$$\begin{aligned}y(1-b)&=(1-a^n[1-p(1-a^{2n})]^{-1}pb^n)(1-b^{2n})+a^nxb^n(1-b)\\
&=1-b^{2n}-a^n[1-p(1-a^{2n})]^{-1}p(1-a^{2n})b^n+a^nx(1-a)b^n\\
&=1-a^npb^{n}-a^np[1-p(1-a^{2n})]^{-1}(1-a^{2n})b^n\\
&=1-a^np\{1+[1-p(1-a^{2n})]^{-1}p(1-a^{2n})\}b^n\\
&=1-a^np[1-p(1-a^{2n})]^{-1}b^n,
\end{aligned}
$$
and so
$$\begin{aligned}y(1-b)^2&=1-b-a^np[1-p(1-a^{2n})]^{-1}b^n(1-b)\\
&=1-b-a^np[1-p(1-a^{2n})]^{-1}(1-a)b^n\\
&=1-b-a^n(1-a)p[1-p(1-a^{2n})]^{-1}b^n\\
&=1-b-(1-b)a^np[1-p(1-a^{2n})]^{-1}b^n\\
&=(1-b)y(1-b).
\end{aligned}
$$
(ii) Since $p$ commutes with $a$, we have
$$(a^nxb^n)a^np[1-p(1-a^{2n})]^{-1}b^n=a^nxpa^{2n}[1-p(1-a^{2n})]^{-1}b^n=0.$$
Thus
$$\begin{aligned}&~~~~y^2(1-b)-y\\
&=ya^np[1-p(1-a^{2n})]^{-1}b^n\\
&=(a^n[1-p(1-a^{2n})]^{-1}pb^n)(1+\sum\limits_{i=1}^{2n-1}b^i)a^np[1-p(1-a^{2n})]^{-1}b^n-\\
&~~~~(1+\sum\limits_{i=1}^{2n-1}b^i)a^np[1-p(1-a^{2n})]^{-1}b^n\\
&=(1+\sum\limits_{i=1}^{2n-1}b^i)a^np[1-p(1-a^{2n})]^{-2}a^{2n}b^n-(1+\sum\limits_{i=1}^{2n-1}b^i)a^np[1-p(1-a^{2n})]^{-1}b^n\\
&=(1+\sum\limits_{i=1}^{2n-1}b^i)a^np[1-p(1-a^{2n})]^{-2}[pa^{2n}-p+p(1-a^{2n})]b^n\\
&=0.
\end{aligned}
$$
(iii) From the proof of part (i) of, it follows that
$$(1-b)y(1-b)-(1-b)=a^n(1-a)p[1-p(1-a^{2n})]^{-1}b^n.$$
Therefore
$$\begin{aligned}&~~~~r[(1-b)y(1-b)-(1-b)]=r(a^n(1-a)p[1-p(1-a^{2n})]^{-1}b^n)\\
&\leq r([1-p(1-a^{2n})]^{-1}b^na^n(1-a)p)=r([1-p(1-a^{2n})]^{-1}a^{2n})r((1-a)p)=0,
\end{aligned}
$$
which implies that $(1-b)y(1-b)-(1-b)$ is quasi-nilpotent. The proof of the opposite implication is parallel to that above.
\end{proof}

\begin{remark} \emph{In recent paper [\cite{Peng-Zhang}], Peng and Zhang established the formula for (generalized) Drazin inverses of $1-a$ and $1-b$ under the conditions $ab^n=b^{n+1}$ and $ba^n=a^{n+1}$. By our Proposition \ref{1.4}, Theorems \ref{4.2} and \ref{4.5}, we can derive new expressions for (generalized) Drazin inverses of $1-a$ and $1-b$, which differ from those of Peng and Zhang. Moreover, Theorems \ref{4.2} and \ref{4.5} lead to the following result: if $a,~b,~c,~d\in \mathcal{A}$ satisfy
\begin{gather}\label{eq3.2}
ac(db)^n=(db)^{n+1} \makebox{~~and~~} db(ac)^n=(ac)^{n+1}, \tag{3.2}
\end{gather}
 then $1-ac$ is one-sided (resp. generalized) Drazin inverse if and only if $1-bd$ is one-sided (resp. generalized) Drazin inverse. Although these conditions (\ref{eq3.2}) cover \textquotedblleft$acd=dbd$\textquotedblright~ and \textquotedblleft$dba=aca$\textquotedblright, our work in Section 3 provides different computational formulas of one-sided (generalized) Drazin inverses.}
\end{remark}

\begin{corollary} \label{4.7} Let $a, b\in \mathcal{A}$ satisfy $ab^n=b^{n+1}$ and $ba^n=a^{n+1}$ for some positive integer $n$. Then
$$\sigma_*(a)\setminus \{0\}=\sigma_*(b)\setminus \{0\},$$
where $*$=$\{$ld, rd, lg, rg, lgd, rgd$\}$.
\end{corollary}

\section{Further results on operators in Banach spaces}

In this section, $\mathcal{B}(X)$ denotes the set of all the bounded linear operators acting on Banach space $X$. Let $T\in \mathcal{B}(X)$, $\mathcal{N}(T)$ and $\mathcal{R}(T)$ denote null space and range of $T$ respectively. The following B-Fredholm theory is a generalization of the classical Fredholm theory. We define the set
$$\bigtriangleup(T):=\{n\in\mathbb{N}:~ \makebox{for~all~}m\geq n,~
\mathcal{N}(T)\cap\mathcal{R}(T^n)=\mathcal{N}(T)\cap\mathcal{R}(T^m)\}.$$
Then, the degree of \textit{stable iteration} $\bigtriangleup(T)$ is defined as dis$(T):$=inf$\bigtriangleup(T)$ (with dis$(T)=\infty$ if $\bigtriangleup(T)=\emptyset$). An operator $T\in \mathcal{B}(X)$ is termed \textit{B-Fredholm} (or \textit{upper semi B-Fredholm}, \textit{lower semi B-Fredholm}), if there exits an integer $n$ such that $\mathcal{R}(T^n)$ is closed and $T|_{\mathcal{R}(T^n)}$ is Fredholm (or \textit{upper semi Fredholm}, \textit{lower semi Fredholm}). The \textit{index} of a B-Fredholm operator $T$ is given by ind$(T)$:=ind$(T|_{\mathcal{R}(T^n)})$, i.e.
$$\makebox{ind}(T)=\makebox{dim}[\mathcal{N}(T)\cap \mathcal{R}(T^n)]-\makebox{dim}X/[\mathcal{R}(T)+\mathcal{N}(T^n)].$$
 For further insights into B-Fredholm theory, readers are encouraged to consult the noteworthy works [\cite{AB1}, \cite{Berkani1}, \cite{Berkani2}, \cite{Berkani3}]. Recall that $T\in\mathcal{B}(X)$ is termed to be \textit{left B-Fredholm} (or \textit{left essentially Drazin invertible}) if $d:=$dis$(T)$ is finite, $\mathcal{N}(T)\cap\mathcal{R}(T^d)$ is finite dimension and $\mathcal{R}(T)+\mathcal{N}(T^d)$ is topologically complemented in $X$. Dually, $T\in \mathcal{B}(X)$ is termed \textit{right B-Fredholm} (or \textit{right essentially Drazin invertible}) if $d:=$dis$(T)$ is finite, $\mathcal{R}(T)+\mathcal{N}(T^d)$ is closed of finite codimension and $\mathcal{N}(T)\cap\mathcal{R}(T^d)$ is topologically complemented in $X$. These one-sided B-Fredholm operators were recently introduced by Ghorbel, \v{Z}ivkovi\'{c}-Zlatanovi\'{c} and Djordjevi\'{c} in [\cite{Ghorbel}, \cite{ZZ}]. In [\cite{Caradus}], Atkinson type Theorems for one-sided Fredholm operators are presented. By [\cite{AB}, Theorems 3.6 and 3.7], Hamdan and Berkani obtain the Atkinson type Theorems for one-sided B-Fredholm operators. Now let $\pi: B(X)\rightarrow B(X)/F(X)$ be the canonical map, where $F(X)$ denotes the set of all the finite rank operators in the Banach space $X$.

\begin{theorem} \emph{([\cite{AB}])} \label{5.1} Let $T\in \mathcal{B}(X)$. Then the following statements hold.

\emph{(i)} $T$ is left B-Fredholm if and only if $\pi(T)$ is left Drazin invertible.

\emph{(ii)} $T$ is right B-Fredholm  if and only if $\pi(T)$ is right Drazin invertible.
\end{theorem}

It is noteworthy that all the results concerning one-sided Drazin inverses in this article are valid in the context of rings with unit (not only in Banach algebras). Now, for $T\in \mathcal{B}(X)$, we define left and right B-Fredholm spectra as follows
$$\sigma_{BF}(T):=\{\lambda\in \mathbb{C}: \lambda-T \makebox{~is~not~B-Fredholm}\},$$
$$\sigma_{LBF}(T):=\{\lambda\in \mathbb{C}: \lambda-T \makebox{~is~not~left~B-Fredholm}\},$$
and
$$\sigma_{RBF}(T):=\{\lambda\in \mathbb{C}: \lambda-T \makebox{~is~not~right~B-Fredholm}\},$$
respectively. By Corollaries \ref{310} and \ref{5.1}, we immediately obtain the following results corresponding to left and right B-Fredholm spectra. These can be regarded as generalizations of classical results concerning spectra of left and right Fredholm or left and right essentially invertible operators, see page 20 of [\cite{Caradus}].

\begin{corollary} \label{5.2} Let $T, S\in \mathcal{B}(X)$. Then
$$\sigma_{*}(TS)\setminus\{0\}=\sigma_{*}(ST)\setminus\{0\},$$
where $*$=$\{$LBF,~RBF,~BF$\}$.
\end{corollary}

Moreover, by Corollary \ref{5.1} and Theorem \ref{4.2}, we deduce the subsequent correlation between the B-Fredholm spectra of two intertwining operators.

\begin{corollary} \label{5.3} Let $T, S\in \mathcal{B}(X)$ satisfy $ST^n=T^{n+1}$ and $TS^n=S^{n+1}$. Then
$$\sigma_{*}(T)\setminus\{0\}=\sigma_{*}(S)\setminus\{0\},$$
where $*$=$\{$LBF,~RBF,~BF$\}$.
\end{corollary}

\noindent {\Large\textbf{Acknowledgements}}
\\
\\
\\
\\
{\Large\textbf{Data Availability Statement}}
\\
\\
No dataset was generated or analyzed during this study.
\\
\\
{\Large\textbf{Declarations}}
\\
\\
\textbf{Conflict of interest} The author declares that they have no known competing financial interests or personal relationships that could have appeared to influence the work reported in this paper. Furthermore, the author pledges to adhere to academic integrity and ethical standards in the research.




\end{document}